\documentclass[11pt]{amsart}
\usepackage{amssymb, amscd}
\usepackage{latexsym}
\usepackage[all,dvips]{xy}
\numberwithin{equation}{section}

\newcommand\eprint[1]{Eprint:~\texttt{#1}}

\newcommand{\QQ}{\mathbb Q}
\newcommand{\FF}{\mathbb F}
\newcommand{\GG}{\mathbb G}
\newcommand{\ZZ}{\mathbb Z}

\newcommand{\PP}{\mathbb P}
\newcommand{\A}{\mathcal A}

\newcommand{\F}{\mathcal F}

\newcommand{\Vc}{\overline{\mathcal V}}

\newtheorem{theorem}{Theorem}[section]
\newtheorem{lemma}[theorem]{Lemma}

\newtheorem{corollary}[theorem]{Corollary}

\newtheorem{definition-lemma}[theorem]{Definition-Lemma}

\theoremstyle{definition}
\newtheorem{definition}[theorem]{Definition}

\theoremstyle{remark}

\newtheorem{assumption}[theorem]{Assumption}

\begin{document}

\title[A stratification on the moduli of K3 surfaces]{A stratification on the moduli of K3 surfaces in positive characteristic}
\author{Gerard van der Geer}
\address{Korteweg-de Vries Instituut, Universiteit van
Amsterdam,
\newline Postbus 94248, 1090 GE Amsterdam, The Netherlands.}
\email{geer@science.uva.nl}

\begin{abstract} 
We review the results on the cycle classes of the strata 
defined by the height and the Artin invariant
on the moduli of K3 surfaces in positive characteristic 
obtained in joint work with Katsura and Ekedahl. 
In addition we prove a new irreducibility result 
for these strata.
\end{abstract}
\maketitle
\hfill{\sl Dedicated to the memory of Fritz Hirzebruch}
\section{Introduction} \label{sec-intro}
Moduli spaces in positive characteristic often possess stratifications for which we do not know 
characteristic $0$ analogues. A good example is the moduli space of elliptic curves in characteristic $p>0$.
If $E$ is an elliptic curve over an algebraically closed field $k$ 
of characteristic $p$, then multiplication
by $p$ on $E$ factors as 
$$
\times p :  E \stackrel{F} {\longrightarrow} E^{(p)} 
\stackrel{V}{\longrightarrow} E ,
$$
where Frobenius $F$ is inseparable and Verschiebung $V$ can be 
separable or inseparable. If $V$ is separable, then $E$ is called ordinary,
while if $V$ is inseparable $E$ is called supersingular. 
In the moduli space $\A_{1} \otimes k$ of elliptic
curves over $k$ there are finitely 
many points corresponding to supersingular elliptic curves, 
and a well-known formula by Deuring, dating from 1941, 
gives their (weighted) number:
$$
\sum_{E/k\,  {\rm supers.}\\ /\cong_k} \frac{1}{\# {\rm Aut}_k(E)} = \frac{p-1}{24}\, ,
$$
where the sum is over the isomorphism classes of supersingular elliptic curves and each curve is counted 
with a weight. We thus find a stratification of the moduli ${\A}_1\otimes {\FF}_p$ 
of elliptic curves with two strata: the ordinary stratum and the supersingular stratum. 
This stratification generalizes to the moduli of principally polarized abelian varieties of
dimension $g$ in positive characteristic where it leads to two stratifications, the Ekedahl-Oort
stratification with $2^g$ strata and the Newton-polygon stratification. 
These stratifications have been the
focus of much study in recent years (see for example 
\cite{O, vdG, E-vdG1, Mo, R, C-O}). 
The dimension of these strata are known and in the case of 
the Ekedahl-Oort stratification we also know by [E-vdG1] 
the cycle classes of these strata 
in the Chow groups of a suitable compactification. 
The formulas for such cycle classes can be seen as a
generalization of the formula of Deuring.

Besides abelian varieties, K3 surfaces form another 
generalization of elliptic curves. The stratification
on the moduli of elliptic curves in positive characteristic 
generalizes to a stratification of the moduli ${\F}_g$ of primitively 
polarized K3 surfaces  of degree $2g-2$ in positive 
characteristic. In fact, in the 1970s Artin and Mazur 
obtained in \cite{A-M} 
an invariant of K3 surfaces by looking at the formal Brauer group 
of a K3 surface. 
For an elliptic curve the distinction between ordinary and supersingular can be formulated 
by looking at the formal group, that is, 
the infinite infinitesimal neighborhood of the origin with the
inherited group law. If $t$ is a local parameter at the origin, then multiplication by $p$ is given by
$$
[p] \, t = a\, t^{p^h}+ \text{higher order terms} \eqno(1)
$$
with $a \neq 0$. Since multiplication by $p$ on $E$ is of degree $p^2$ 
and inseparable, we have  $1\leq h \leq 2$ 
and $h=1$ if $E$ is ordinary and $h=2$ if $E$ is supersingular.
The formal group allows a functorial description as the functor on spectra $S$ 
of local Artin $k$-algebras with residue field $k$
given by 
$$
S \mapsto \ker \{ H^1_{\rm et} (E\times S, {\GG}_m) \to H^1_{\rm et} (E,{\GG}_m)\},
$$
where $H^1_{\rm et}(E,{\GG}_m)\cong H^1(E,{\mathcal O}_E^*)$ classifies line bundles 
on $E$. The invariant of Artin and Mazur generalizes this. For a K3 surface $X$ they looked
at the functor of local Artinian schemes over $k$ 
with residue field $k$ given by
$$
S \mapsto \ker \{ 
H^2_{\rm et} (X\times S, {\GG}_m) \to H^2_{\rm et} (X,{\GG}_m)\},
$$
and showed that it is representable by a formal Lie group,
called the formal Brauer group. Its tangent space is $H^2(X,{\mathcal O}_X)$, so we 
have a $1$-dimensional formal group.
Now over an algebraically closed field $1$-dimensional formal groups are classified by their height: 
in terms of a local coordinate $t$ multiplication by $p$ is either zero 
or takes the form 
$[p]\, t = a \, t^{p^h}+ $ higher order terms, with $a \neq 0$. 
If multiplication by $p$ vanishes we say $h=\infty$, 
and then we have 
the formal additive group $\hat{\GG}_a$ and if $h<\infty$ 
we have a $p$-divisible 
formal group. 

Artin and Mazur connected this invariant $h(X)$ to the geometry 
of the K3 surface
by proving that if $h(X) \neq \infty$ then
$$
\rho(X) \leq 22-2h(X), \eqno(2)
$$
where $\rho(X)$ is the Picard number of $X$. In particular, 
either we have $\rho=22$ (and then necessarily $h=\infty$), or $\rho \leq 20$.  

The case that $\rho=22$ can occur in positive characteristic as Tate observed: 
for example, in characteristic $p\equiv 3 \, (\bmod 4)$ 
the Fermat surface $x^4+y^4+z^4+w^4=0$ has $\rho=22$ (see \cite{T,S}). 

If $h(X)=\infty$ then the K3 surface $X$ is called supersingular. 
By the result of Artin and Mazur a K3 surface $X$ with $\rho(X)=22$ 
must be supersingular. In 1974 Artin conjectured the converse (\cite{A}): 
a supersingular K3 surface has $\rho(X)=22$. 
This has now been proved by Maulik,
Charles and Madapusi Pera for $p\geq 3$, see \cite{M,C,P2}.
In the 1980s Rudakov, Shafarevich and Zink proved that supersingular K3 surfaces 
with a polarization of degree $2$ have $\rho=22$ in characteristic 
$p\geq 5$, see
\cite{R-S-Z}.

The height is upper semi-continuous in families. 
The case $h=1$ is the generic case; in particular the K3 surfaces with
$h=1$ form an open set. By the inequality (2) we have
$$
1 \leq h \leq 10 \qquad {\rm or} \quad h=\infty.
$$ 
In the moduli space 
${\F}_g$ of primitively polarized  K3 surfaces of genus $g$ (or equivalently, of
degree $2g-2$) with $2g-2$ prime to $p$, the locus of K3 
surfaces with height $\geq h$ is locally closed and
has codimension $h-1$ and we thus have 
$11$ strata in the $19$-dimensional moduli space ${\F}_g$.
The supersingular locus has dimension $9$. Artin showed that it is
further stratified by the Artin invariant $\sigma_0$: assuming that $\rho=22$
one looks at 
the N\'eron-Severi group ${\rm NS}(X)$ with its intersection pairing; 
it turns out that the discriminant group
${\rm NS}(X)^{\vee}/{\rm NS(X)}$ is an elementary $p$-group isomorphic to
$({\ZZ}/p{\ZZ})^{2\sigma_0}$ and one thus
obtains another invariant.
The idea behind this 
is that, though $\rho=22$ stays fixed, divisor classes in the limit
might become divisible by $p$, thus changing the N\'eron-Severi lattice and
$\sigma_0$. The invariant $\sigma_0$ is lower semi-continuous.
The generic case (supersingular) is where
$\sigma_0=10$ and the most degenerate case is the so-called superspecial 
case $\sigma_0=1$. 

In total one obtains a stratification on the moduli space ${\mathcal F}_g$ 
of K3 surfaces with a
primitive polarization of genus $g$ 
with $20$ strata $V_j$ with $1 \leq j \leq 20$ 
$$
\overline{V}_j=
\{ [X]\in {\mathcal F}_g : h(X) \geq j\} \qquad  1\leq j \leq 10, 
$$
and
$$
\overline{V}_j=
\{ [X] \in {\mathcal F}_g : h(X)=\infty, \,  \sigma_0(X) \leq 21-j\} \qquad 
 11\leq j \leq 20 \, ,
$$
the closures $\overline{V}_j$ of which
 are linearly ordered by inclusion. 
In joint work with Katsura \cite{vdG-K1} we determined the
cycle classes of the strata 
$\overline{V}_j$ ($j=1,\ldots,10$) of height $h\geq j$ 
$$
[V_j]= (p-1)(p^2-1)\cdots (p^{j-1}-1)\, \lambda_1^{j-1}, \qquad 1\leq j \leq 10 
\eqno(3)
$$
where 
$\lambda_1= c_1(\pi_* (\Omega^2_{\mathcal X / {\F}_g}))$
is the Hodge class
with $\pi: {\mathcal X} \to {\F}_g$ the universal family.
We also determined the class of the supersingular locus $\overline{V}_{11}$.
Moreover, we proved that the singular locus of $\overline{V}_j$ is
contained in the stratum of the supersingular 
locus where the Artin invariant is at most $j-1$, see \cite[Thm.\ 14.2]{vdG-K1}. Ogus made this more
precise in \cite{Og2}. For more on the moduli of supersingular K3 surfaces
we also refer to \cite{Og1, R-S2,L}.
 
But the
cycle classes of the other strata $\overline{V}_j$ 
for $j=12,\ldots, 20$ given by the Artin
invariant turned out to be elusive.
In joint work with Ekedahl \cite{E-vdG2} we developed a 
uniform approach by applying the philosophy of \cite{vdG,E-vdG1} 
of interpreting
these stratifications in terms of flags on the cohomology and 
eventually were able to determine all cycle classes. 
All these cycle classes
are multiples of powers of the Hodge class 
$\lambda_1$.

Our approach uses flags on the de Rham cohomology, 
here on $H^2_{\rm dR}$ as opposed to $H^1_{\rm dR}$
for abelian varieties. The space $H^2_{\rm dR}(X)$ is 
provided with a non-degenerate intersection form 
and it carries a filtration, the Hodge filtration.
But in positive characteristic it carries a second filtration 
deriving from the fact that we do not have a Poincar\'e lemma, 
or in other words, it derives from the Leray spectral sequence 
applied to the relative Frobenius morphism $X \to X^{(p)}$.
See later for more on this so-called conjugate filtration.
We thus find two filtrations on $H^2_{\rm dR}(X)$ and these are not
necessarily transversal. We say that $X$ is {\sl ordinary} if the 
two filtrations are transversal and that $X$ is {\sl superspecial} 
if the two filtrations coincide.
These are two extremal cases, but by considering the relative position
of flags refining the two flags one obtains a further discrete invariant 
and one 
retrieves in a uniform way the invariants encountered above, 
the height~$h$ and the Artin invariant~$\sigma_0$.

For applications it is important that we consider 
moduli of K3 surfaces together with an embedding of a 
non-degenerate lattice $N$ in the N\'eron-Severi group of $X$
such that it contains a semi-ample class of degree prime 
to the characteristic $p$,
and then look at the primitive part of the de Rham cohomology. 
If the dimension $n$ of this primitive cohomology is even, 
this forces us to deal with very subtle questions related to the
distinction of orthogonal group ${\rm O}(n)$ versus the special orthogonal 
group ${\rm SO}(n)$. 

Instead of working directly on the
moduli spaces of K3 surfaces, we work on the space of flags on the primitive
part of $H^2_{\rm dR}$, that is, we work on a flag bundle over the moduli space.
The reason is that the strata that are defined on this
space are much better behaved than the strata on the moduli 
of K3 surfaces itself.
In fact, up to infinitesimal order $p$ the strata on the flag space over ${\F}_g$ 
look like the strata (the Schubert cycles) on the flag space for the orthogonal group.
These strata are indexed by elements of a Weyl group.

In order to get the cycle classes of the strata on the moduli of K3 surfaces
we note that these latter strata are linearly ordered. This allows us to
apply fruitfully a Pieri type formula which expresses the intersection product
of a cycle class with a first Chern class (the Hodge class in our case) as
a sum of cycle classes of one dimension less. 

We apply this on the flag space and
then project it down. There are many more strata on the flag space of the
primitive cohomology than on the moduli space. Some of these strata, the so-called
final ones, map in an \'etale way to their image in the moduli space; for the
non-final ones, either the image is lower-dimensional, and hence its cycle class 
can be ignored, or the map is inseparable and factors through a final stratum
and the degree of the inseparable map can be calculated. In this way one arrives
at closed formulas for the cycle classes of the strata on the moduli space.

\smallskip
  
We give an example of the formula for the cycle classes from
\cite{E-vdG2} in the following special case. 
Let $p>2$ and $\pi: {\mathcal X} \to {\F}_g$ be the universal family of
K3 surfaces with a primitive polarization 
of degree $d=2g-2$ with $d$ prime to $p$. Then there are $20$
strata on the $19$-dimensional 
moduli space ${\F}_g$ parametrized by so-called final elements
$w_i$ with $1 \leq i \leq 20$ in the Weyl group of ${\rm SO}(21)$. 
These are ordered by their length $\ell(w_i)$ (in the sense of length in
Weyl groups)
starting with the
longest one.
The strata ${\mathcal V}_{w_i}$ for $i=1,\ldots 10$ are the strata
of height $h= i$, the stratum ${\mathcal V}_{w_{11}}$ is the supersingular stratum,
while the strata ${\mathcal V}_{w_i}$ for $i=11,\ldots,20$ are the strata where
the Artin invariant satisfies $\sigma_0 = 21-i$. 

\begin{theorem}
The cycle class of the closed stratum $\overline{\mathcal V}_{w_i}$ on the moduli space ${\F}_g$
is given by
\begin{eqnarray*}
{\rm i)} \quad 
[\Vc_{w_k}] &=& (p-1)(p^2-1)\cdots(p^{k-1}-1) \lambda_1^{k-1} \quad
\hbox{\rm if $1\leq k\leq 10$,}\\
{\rm ii)} \quad [\Vc_{w_{11}}] &=&\frac{1}{2} (p-1)(p^2-1)\cdots(p^{10}-1) 
\lambda_1^{10},\\
{\rm iii)} \quad 
[\Vc_{w_{10+k}}] &=&\frac{1}{2}
\frac{(p^{2k}-1)(p^{2(k+1)}-1)\cdots(p^{20}-1)}{(p+1)\cdots(p^{11-k}+1)}
\lambda_1^{9+k} \quad
\hbox{\rm if $2\leq k\leq 10$.}
\end{eqnarray*}
\end{theorem}

Here $\lambda_1=c_1(L)$ with $L=\pi_* (\Omega^2_{{\mathcal X}/{\mathcal F}_g})$ 
is the Hodge class. Sections of $L^{\otimes r}$ correspond to modular forms
of weight $r$.
It is known (cf.\ \cite{vdG-K2} ) that the  class
$\lambda_1^{18}\in {\rm CH}_{\QQ}^{18}({\mathcal F}_g)$ 
vanishes on ${\mathcal F}_g$. But the formulas
can be made to work also on the closure of the image ${\mathcal F}_g$ 
embedded in projective
space by the sections of a sufficiently high power of $L$, 
so that the last two formulas (involving $\lambda_1^{18}$ and $\lambda_1^{19}$) 
are non-trivial and still make sense. 
Note here that $\lambda_1$ is an ample class;
this is well-known by Baily-Borel in characteristic $0$, but  now
we know it too in characteristic $p\geq 3$ by work
of Madapusi-Pera \cite{P1} and Maulik \cite{M}.

In particular, we can give an explicit formula for 
the weighted number of superspecial K3 surfaces of genus~$g$ by using a 
formula for $\deg(\lambda_1^{19})$ from \cite{G-H-S}.
We consider the situation where we have a primitive polarization of degree
$d=2d'$ inside $N$ with with
$$
N^{\bot}= 2 \, U \bot m \, E_8(-1) \bot \langle -d \rangle \, , \eqno(4)
$$
where $U$ is a hyperbolic plane and $m=0$ or $m=2$.

\begin{theorem} 
The weighted number 
$$
\sum_{X \, {\rm superspecial}/ \cong} \frac{1}{\# {\rm Aut}_k(X)} 
$$
of superspecial K3 surfaces with a primitive polarization $N$ 
of degree $d=2d^{\prime}$ prime to the characteristic $p$ with $N^{\bot}$
as in (1), 
is given by
$$
\frac{-1}{2^{4m+1}}\frac{p^{8m+4}-1}{p+1} \, 
\left( (d^{\prime})^{10} \prod_{\ell|d^{\prime}}  (1+\ell^{-4m-2}) \right) 
\, \zeta(-1)\zeta(-3) \cdots \zeta(-8m-3)
$$
where $\zeta$ denotes the Riemann zeta function and $\ell$ runs over the primes
dividing $d'$.
\end{theorem}
For $m=0$ this formula can be applied to count the number of Kummer surfaces
coming from superspecial principally polarized abelian surfaces and the
formula then agrees with the formulas of \cite{E, vdG}.

Formulas like those given in (3) and in Theorem 1.1 for the classes 
of the height  strata were obtained in joint work with Katsura
\cite{vdG-K1} by different (ad hoc) methods using formal groups and Witt vector
cohomology; but these methods did not suffice to calculate the cycle classes
of the Artin invariant strata.

A simple corollary is (see \cite[Prop.\ 13.1]{E-vdG2}).

\begin{corollary} A supersingular (quasi-)elliptic K3 surface with a section 
cannot have Artin invariant $\sigma_0=10$.
\end{corollary}

This result was obtained independently by Kondo and Shimada using a 
different method in \cite[Cor.\ 1.6]{K-S}.

In addition to reviewing the results from \cite{E-vdG2} 
we prove irreducibility results  for the strata; about half 
of the strata on the moduli space ${\mathcal F}_g$ are 
shown to be irreducible. Here we use the local structure of the 
strata on the flag space.

\begin{theorem} Let $p\geq 3$ prime to the degree $d=2g-2$.
For a final element $w \in W_m^B$ (resp.\ $w \in W_m^D)$ of length $\ell(w)\geq m$ the stratum
$\overline{\mathcal V}_w$ in ${\mathcal F}_g$ is irreducible.
\end{theorem}

So the strata above the supersingular locus are all irreducible.
We have a similar result in ${\mathcal F}_N$.

The formulas we derived deal with the group ${\rm SO}(n)$; 
for K3 surfaces we can restrict $n \leq 21$, but the formulas for larger $n$ 
might find  applications to the moduli of hyperk\"ahler varieties in positive 
characteristic (by looking at the middle dimensional de Rham cohomology
or at $H^2_{\rm dR}$ equiped with the Beauville-Bogomolov form).

\section{Filtrations on the de Rham cohomology of a K3 surface}
Let $X$ be a K3 surface over an algebraically closed field of 
characteristic $p>2$ and let $N \hookrightarrow {\rm NS}(X)$ be an isometric embedding of a non-degenerate 
lattice in the N\'eron-Severi group ${\rm NS}(X)$ (equal to the Picard group for a
K3 surface) 
and assume that $N$ contains a semi-ample line bundle and that the discriminant
of $N$ is coprime with $p$ (that is,
$p$ does not divide $\# N^{\vee}/N$). 
We let $N^{\bot}$ be the primitive cohomology, 
that is, the orthogonal complement of
the image of $c_1(N)$ of $N$ in $H^2_{\rm dR}(X)$.
It carries a Hodge filtration 
$$
0 = U_{-1} \subset U_0 \subset U_1 \subset U_2=N^{\bot}
$$
of dimensions $0,1,n-1,n$ and comes with a non-degenerate intersection form 
for which the Hodge filtration is self-dual: $U_0^{\bot}=U_1$. Now in
positive characteristic we have another filtration
$$
0 = U_{-1}^c \subset U_0^c \subset U_1^c \subset U_2^c =N^{\bot} \, ,
$$
the conjugate filtration; it is self-dual too.
The reason for its existence 
is that the Poincar\'e lemma does not hold in positive
characteristic. If $F: X \to X^{(p)}$ is the relative Frobenius
morphism then we have a canonical (Cartier) isomorphism
$$
C: {\mathcal H}^j(F_*\Omega_{X/k}^{\bullet}) \cong \Omega^j_{X^{(p)}/k}
$$
and we get a non-trivial spectral sequence from this:
the second spectral sequence of hypercohomology 
with $E_2$-term 
$E^{ij}_2=H^i(X^{(p)},{\mathcal H}^j(\Omega^{\bullet}))$
which by the inverse Cartier isomorphism $C^{-1}: \Omega^j_{X^{(p)}} \simeq 
{\mathcal H}^j(F_*(\Omega^{\bullet}_{X/k}))$ 
can be rewritten as
$H^i(X^{(p)},\Omega^j_{X^{(p)}/k})$,
degenerating at the $E_2$-term and abutting to $H^{i+j}_{\rm dR}(X/k)$. 
This leads to a second filtration on the de Rham cohomology. 

The inverse Cartier operator gives an isomorphism 
$$
F^*(U_i/U_{i-1})\cong U_{2-i}^c/U^c_{1-i}\, .
$$ 
As a result we have two (incomplete) flags 
forming a so-called F-zip in the sense of \cite{M-W}.
Unlike the characteristic zero situation
 where the Hodge flag and its complex conjugate are
transversal, the two flags in our situation are not 
necessarily transversal. In fact, the 
K3 surface $X$ is called ordinary if these flags are transversal 
and superspecial if they coincide. These are just two extremal 
cases among more possibilities.

Before we deal with these further possibilities,
we recall some facts about isotropic flags in a
non-degenerate orthogonal space. Let $V$ be a non-degenerate orthogonal
space of dimension $n$ over a field of characteristic $p>2$. We have
to distinguish the cases $n$ odd and $n$ even, the latter being more subtle.
We look at isotropic flags 
$$
(0) =V_0 \subset V_1 \subset \cdots \subset V_r
$$
with $\dim V_i=i$ in $V$, that is, 
we require that the intersection form vanishes on $V_r$.
We call the flag maximal if $r=[n/2]$. We can complete a maximal flag by
putting $V_{n-j}=V_j^{\bot}$. Now if $n=2m$ is even, a complete isotropic 
flag $V_{\bullet}$ determines another complete isotropic flag by putting
$V_i^{\prime}=V_i$ for $i < n/2$
and by taking for $V_m^{\prime}$ the unique maximal isotropic space 
containing $V_{m-1}$ but different from $V_m$. We call this flag 
$V_{\bullet}^{\prime}$ the {\sl twist} of $V_{\bullet}$.

In fact, if $n$ is even the group ${\rm SO}(n)$ does not act transitively
on complete flags.

Given two complete isotropic flags their relative position is given by 
an element of a Weyl group. If $n$ is odd we let $W_m^B$ be the Weyl group
of ${\rm SO}(2m+1)$. It can be identified with the following 
subgroup of the symmetric group $\frak{S}_{2m+1}$
$$
\{ \sigma \in \mathfrak{S}_{2m+1} : \sigma(i)+\sigma(2m+2-i)=2m+2 
\text{ for all $1\leq i \leq 2m+1$} \}.
$$ 
The fact is now that the ${\rm SO}(2m+1)$-orbits of pairs of totally
isotropic complete flags are in 1-1 correspondence 
with the elements of $W_m^B$ given by
$$
w \longleftrightarrow \big(\sum_{j\leq i} k\cdot e_j, \sum_{j\leq i} 
k\cdot e_{w^{-1}(j)} \big)
$$
with the $e_i$ a fixed orthogonal basis with $\langle e_i, e_j \rangle 
=\delta_{i,2m+2-j}$. The simple reflections $s_i\in W_m^B$
for $i=1,\ldots,m$ are given by $s_i=(i, i+1)(2m+1-i, 2m+2-i)$ if $i<m$
and by $s_m=(m, m+2)$, and will play an important role here.

But in the case that $n=2m$ is even we have to replace the Weyl group $W_m^C$
(of ${\rm O}(2m)$) given by  
$$
\big\{ \sigma \in \mathfrak{S}_{2m} : \sigma(i)+\sigma(2m+1-i)=2m+1
\text{ for all $1\leq i \leq 2m$}\big\}
$$
by the index $2$ subgroup $W_m^D$ given by the extra parity condition
$$
\# \{ 1 \leq i \leq m : \sigma(i)>m\} \equiv \, 0 \, (\bmod 2) .
$$
The simple reflections $s_i\in W_m^D$ 
are given by $s_i=(i, i+1)(2m-i, 2m+1-i)$
for $i<m$ and by $s_m=(m-1,m+1)(m, m+2)$.
In the larger group $W_m^C$ we have the simple reflections $s_i$ with $1\leq i \leq m-1$ and $s_m^{\prime}=(m, m+1)$.  
Note that  $s_m^{\prime}$ commutes with the $s_i$ for $i=1,\ldots,m-2$
and conjugation by it interchanges $s_{m-1}$ and $s_m$.

The ${\rm SO}(2m)$-orbits of pairs of totally isotropic complete flags 
are in bijection with the elements of $W_m^C$ given by
$$
w \longleftrightarrow \big(\sum_{j\leq i} k\cdot e_j, \sum_{j\leq i} 
k\cdot e_{w^{-1}(j)} \big)
$$
with basis $e_i$ with $\langle e_i,e_j\rangle =\delta_{i,2m+1-j}$.
Twisting the first (resp.\ second) flag corresponds to changing $w$
to $ws_m'$ (resp.\ $s_m'w$). 

Back to K3 surfaces. 
We can refine the conjugate flag on $H^2_{\rm dR}(X)$ 
to a full (increasing) flag $D^{\bullet}$ 
and use the Cartier  operator to transfer it back to a decreasing flag 
$C^{\bullet}$ on the Hodge filtration $U_{\bullet}$. 
We thus get two full flags.

\begin{definition} A full flag refining the conjugate filtration is called
stable if $D_j\cap U_i + U_{i-1}$ is an element of the 
$C^{\bullet}$ filtration or of its twist. 
A flag is called {\sl final} if it is stable and complete.
\end{definition}

Final flags correspond to so-called {\sl final elements} in the Weyl group
defined as follows.
Elements in the Weyl group $W_m^B$ (resp.\ $W_{m}^D$) 
which are reduced with respect to the set of
roots obtained after removing the first root (so that the remaining roots form 
a root system of type $B_{m-1}$ (resp.\ $D_{m-1}$)) are called 
final elements.

If $n=2m+1$ is odd we have $2m$ final elements in $W_m^B$. These are
the elements $\sigma$ given by $[\sigma(1),\sigma(2),\ldots,\sigma(m)]$
and we can list these as $w_1=[2m+1,2,3,\ldots]$, $w_2=[2m,1,3,\ldots], \ldots,
w_{2m}=[1,2,\ldots,m]$. These final elements $w_j$ are linearly ordered
by their length $\ell(w_j)=2m-j$, with $w_1$ being the longest element
and $w_{2m}$ equal to the identity element.

If $n=2m$ then we also have $2m$ final elements in $W_m^C$, 
but these are no longer linearly ordered by their length,
but the picture is rather
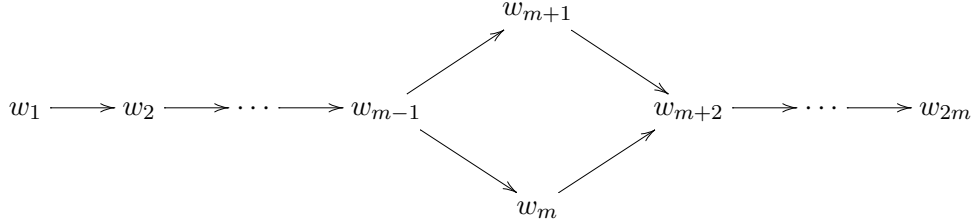
\begin{displaymath}
\begin{xy}
\xymatrix{
   &  &  &   & w_{m+1} \ar[dr] \\
 w_1\ar[r]  & w_2\ar[r]& \cdots   \ar[r] & w_{m-1} \ar[dr]\ar[ur] & & w_{m+2}
\ar[r]  &\cdots \ar[r] & w_{2m}\\
            &  &                         & & w_{m} \ar[ur] \\
}
\end{xy}
\end{displaymath}
Conjugation by $s_m^{\prime}$ interchanges the
two final elements of length $m-1$. This corresponds to twisting a complete
isotropic flag. We shall denote $w_js_m^{\prime}$ by $w_j^{\prime}$ and
we shall call these {\sl twisted final elements}.

The following theorem shows that we can read off the height $h(X)$ of the 
formal Brauer group and the Artin invariant $\sigma_0(X)$ from these final 
filtrations on $H^2_{\rm dR}$.
The following theorem is proven in \cite{E-vdG2}.
Recall that the discriminant of $N$ is assumed to be prime to $p$.

\begin{theorem} 
Let $X$ be a K3 surface with an 
embedding $N \hookrightarrow {\rm NS}(X)$ and 
let $H\subset H^2_{\rm dR}$ be the primitive part of the 
cohomology with $n=\dim(H)$ and $m=[n/2]$. 
Then $H$ possesses a final filtration; all final filtrations 
are of the same combinatorial type $w$. Moreover, 
\begin{enumerate} 
\item[i)] $X$ has finite height $h < n/2$  if and only if $w=w_h$ or $w_h'$.
\item[ii)] $X$ has finite height $h=n/2$ if and only if $w=w_m'$.
\item[iii)] $X$ has Artin invariant $\sigma_0< n/2$  if and only 
if $w=w_{2m+1-\sigma_0}$ or $w= w_{2m+1-\sigma_0}'$.
\item[iv)] $X$ has Artin invariant $\sigma_0=n/2$  
if and only if $w=w_{m+1}$.
\end{enumerate}
\end{theorem}

In case i) (resp.\ in case iii)) 
we can distinguish these cases $w=w_h$ or $w=w_h'$ 
(resp.\ $w=w_{n-\sigma_0}$ or $w=w_{n-\sigma_0}'$) 
for even $n$ by looking whether the 
so-called middle part of the cohomology is split, 
or equivalently, by the sign of the permutation $w$. We get $w_h'$ in
case i) exactly if $w$ is an odd permutation and in case 
iii) we get $w=w_{n-\sigma_0}$ exactly if $w$ is an even permutation.

Looking at the diagram above one sees that the theorem excludes one of
the two possibilities corresponding to the two final elements
$w_m$ and $w_{m+1}$ of length $m$.
This is analyzed in detail in \cite[Section 5]{E-vdG2}. 
It then agrees with the fact that the (closed) strata defined by the 
height and the Artin invariant are linearly ordered by inclusion,
whereas the final $w_i$ in the above diagram are not.

\section{Strata on the flag space}
Suppose that we have a family $f: {\mathcal X} \to S$ of 
$N$-marked K3 surfaces over a smooth base $S$. 
We shall make a versality assumption. At a geometric point $s$ of $S$ 
we have the Kodaira-Spencer map
$T_sS \to H^1(X_s,T_{X_s}^1)$. We have a natural map 
$H^1(X_s,T_{X_s}^1)\to 
{\rm Hom}(H^0(X_s,\Omega^2_{X_s}), H^1(X_s,\Omega^1_{X_s}))$ 
and we can project $H^1(X_s,\Omega^1_{X_s})$ 
to the orthogonal complement $P$  of the image of 
$N\hookrightarrow {\rm NS}(X_s)$ 
in $H^1(X_s,\Omega^1_{X_s})$. 
\begin{assumption}\label{versality}
The versality assumption is the requirement that the resulting map 
$$
T_sS \to {\rm Hom}(H^0(X_s,\Omega^2_{X_s}),P)
$$ 
is surjective.
\end{assumption}
The primitive cohomology forms a vector bundle ${\mathcal H}$ 
of rank $n$ over $S$. It comes with two partial orthogonal isotropic 
flags: the conjugate flag and the Hodge flag.
If we choose a complete orthogonal flag refining the conjugate filtration and
transfer it to the Hodge filtration by the Cartier operator we get
two flags and we can measure the relative position. This defines strata on $S$.
This implies that we have to choose a flag and we are thus forced to work
on the flag space ${\mathcal B}$ over $S$ (or ${\mathcal B}_N$ over
${\mathcal F}_N$)
of complete isotropic flags refining the Hodge filtration. 
(Since we are using ${\F}_N$ for the moduli space of $N$-polarized K3 surfaces
 we use another letter for the flag space; say ${\mathcal B}_N$ for banner). 

To define the strata 
scheme-theoretically we consider the general case of a semi-simple Lie group
$G$ and a Borel subgroup $B$ and a $G/B$-bundle $R\to Y$ over some scheme $Y$
with $G$ as structure group. Let  $r_i: Y\to R$ ($i=1,2$) be two sections.
If $w$ is an element of the Weyl group $W$ of $G$ we define a locally closed
subscheme ${\mathcal U}_w$ of $Y$ as follows. We choose locally (possibly
in the \'etale topology) a trivialization of $R$ such that $r_1$ is a constant
section. Then $r_2$ corresponds to a map $Y \to G/B$ and we define
${\mathcal U}_w$ (resp.\ $\overline{\mathcal U}_w$) 
to be the inverse image of the
$B$-orbit $BwB$ (resp.\ of its closure). 

We thus find strata ${\mathcal U}_w$ and $\overline{\mathcal U}_w$ of 
${\mathcal B}_N$; 
it turns out that  $\overline{\mathcal U}_w$ 
is the closure of ${\mathcal U}_w$. 
It might seem that working on the flag space brings us farther
from the goal of defining and studying strata on the base space $S$ or on
the moduli spaces. However, working with the strata on the flag space
has the advantage that the strata are much better behaved there. 

The space ${\mathcal B}_N$ together with the strata ${\mathcal U}_w$ is 
a stratified space. The space ${\mathcal Fl}_n$ of complete self-dual flags
on an orthogonal space $V$ also carries a stratification, namely
 by Schubert cells.
It is fiber space 
over the space of maximal isotropic subspaces ${\mathcal I}_n$.

The main idea is now that our space ${\mathcal B}_N$ over ${\mathcal F}_N$
locally at a point up to the $(p-1)$st infinitesimal neighborhood looks like ${\mathcal Fl}_n
\to {\mathcal I}_n$
at a suitable point as stratified spaces.
This idea was developed in \cite{E-vdG1} and here it profitably can be used too.

If $(R,m)$ is a local ring {\sl the height $1$ hull} of $R$ 
(resp.\ of $S={\rm Spec}(R)$) is $R/m^{(p)}$ (resp.\ ${\rm Spec}(R/m^{(p)})$)
with $m^{(p)}$ the ideal
generated by the $p$th powers of elements of $m$. It defines the 
height $1$ neigborhood of the point given by $m$. We call two local rings
height $1$-isomorphic if their height $1$ hulls are isomorphic. 

\begin{theorem}
Let $k$ be a perfect field of characteristic $p$. For each $k$-point $x$ of 
${\mathcal B}_N$ there exists a $k$-point $y$ of ${\mathcal Fl}_n$ such that
the height $1$ neighborhood of $x$ is isomorphic (as stratified spaces) 
to the height $1$ neighborhood of $y$. 
\end{theorem}

Indeed, we can trivialize the de Rham cohomology with its Gauss-Manin 
connection on the height $1$ neighborhood of $x$ (because the ideal of $x$
has a divided power structure for which divided powers of degree $\geq p$
are zero). This has strong consequences for our strata, cf.\ the 
following result from \cite{E-vdG2}.

\begin{theorem}\label{height1structure}
The strata ${\mathcal U}_w$ on the flag space ${\mathcal B}_N$ satisfy the following properties:
\begin{enumerate}
\item{} The stratum ${\mathcal U}_w$ is smooth of dimension equal to the 
length $\ell(w)$ of~$w$.
\item{} The closed stratum $\overline{\mathcal U}_w$ is reduced, Cohen-Macaulay and normal of dimension $\ell(w)$ and equals the closure of ${\mathcal U}_w$.
\item{} If $w$ is a final element then the restriction of ${\mathcal B}_N
\to {\mathcal F}_N$ to ${\mathcal U}_w$ is a finite surjective \'etale
covering from ${\mathcal U}_w$ to its image ${\mathcal V}_w$.
\end{enumerate}
\end{theorem} 

The degrees of the maps $\pi_w: {\mathcal U}_w \to {\mathcal V}_w$ for final $w$
coincide with the number of final filtrations of type $w$ and these numbers
can be calculated explicitly. For example,
for $w_i\in W_m^B$ with $1 \leq i <m $ we have
$$
\deg \pi_{w_i}/\deg \pi_{w_{i+1}} = p^{2m-2i-1}+p^{2m-2i-2}+\ldots + 1.
$$
\section{The Cycle Classes}
We consider a family of $N$-polarized K3 surfaces ${\mathcal X} \to S$
with $S$ smooth and satisfying the versality assumption \ref{versality}.
Our strategy in \cite{E-vdG2} is to apply inductively a Pieri formula
to the final strata, which expresses the intersection 
$\lambda_1 \cdot [\overline{\mathcal U}_w]$ as a sum over the classes 
$[\overline{\mathcal U}_v]$, where  $v$ is running through the elements of the
Weyl group of the form 
$v=ws_{\alpha}$ with $s_{\alpha}$ simple and
$\ell(ws_{\alpha})=\ell(w)-1$. In fact,
we use a Pieri formula due to Pittie and Ram \cite{P-R}.
In general these elements $v$ of colength $1$ are not final and this 
forces us to analyze what happens with the strata ${\mathcal U}_v$
under the projection ${\mathcal B}_N \to {\mathcal F}_N$.
It turns out that for a non-final stratum either the projection is
to a lower-dimensional stratum or factors through an inseparable
map to a final stratum. The degree of these inseparable maps can be calculated.
In the case of a map to a lower dimensional stratum we can neglect these
for the cycle class calculation. 

So suppose that for an
element $w$ in the Weyl group we have $\ell(ws_i)=\ell(w)-1$
for some $1<i\leq m$. 
This means that if $A_{\bullet}$ and $B_{\bullet}$ denote
the two flags, that the image of $B_{w(i+1)}\cap A_{i+1}$ in $A_{i+1}/A_i$
is $1$-dimensional and thus we can change the flag $A_{\bullet}$ to a flag $A'_{\bullet}$ by setting 
$A'_j=A_j$ for $j\neq i$ and $A_i'/A_{i-1}$ equal to the image of 
$B_{w(i+1)}\cap A_{i+1}$. This gives us a map
$$
\sigma_{w,i} : {\mathcal U}_w \to {\mathcal F}_N, \qquad
(A_{\bullet}, B_{\bullet}) \mapsto (A'_{\bullet},B_{\bullet}).
$$
In this situation the image depends on the length $\ell(s_iws_i)$:

\begin{lemma}
If $\ell(s_iws_i)=\ell(w)$ then the image of $\sigma_{w,i}$ is equal to ${\mathcal U}_{s_iws_i}$ and the map is purely inseparable of degree $p$. If
$\ell(s_iws_i)=\ell(w)-2$ then $\sigma_{w,i}$ maps onto ${\mathcal U}_{ws_i}
\cup {\mathcal U}_{s_iws_i}$ and the map $\sigma_{w,i}$ is not generically finite.
\end{lemma}

We then analyze in detail the colength $1$ elements occuring and whether they give rise to projections that loose dimension or are inseparable to final strata.
This is carried out in detail in \cite[Section 9-11]{E-vdG2}.
In this way the Pieri formula enables us to calculate the cycle classes.

The result for the cycle classes of the strata $\overline{\mathcal V}_{w_i}$
in the case that $n$ is odd (and with $m=[n/2]$) reads (cf.\ \cite{E-vdG2}):
\begin{theorem}\label{Oddcycleclasses}
  The cycle classes of the final strata $\Vc_w$ on the base $S$ are polynomials
  in $\lambda_1$ with coefficients that are polynomials in $\frac{1}{2}\ZZ[p]$
  given by
\begin{eqnarray*}
{\rm i)} \quad
[\Vc_{w_k}] &=& (p-1)(p^2-1)\cdots(p^{k-1}-1) \lambda_1^{k-1} \quad
\hbox{\rm if $1\leq k\leq m$,}\\
{\rm ii)} \quad [\Vc_{w_{m+1}}] &=&\frac{1}{2} (p-1)(p^2-1)\cdots(p^{m}-1)
\lambda_1^{m},\\
{\rm iii)} \quad
[\Vc_{w_{m+k}}] &=&\frac{1}{2}
\frac{(p^{2k}-1)(p^{2(k+1)}-1)\cdots(p^{2m}-1)}{(p+1)\cdots(p^{m-k+1}+1)}
\lambda_1^{m+k-1} \quad
\hbox{\rm if $2\leq k\leq m$.}
\end{eqnarray*}
\end{theorem} 

In the case where $n$ is even the result for the untwisted final 
elements is the following:

\begin{theorem}\label{Untwistedclasses}
The cycle classes of the final strata $\Vc_w$ for final
elements $w=w_j \in W_m^D$ on the base $S$ are
powers of $\lambda_1$ times polynomials
in $\frac{1}{2}\ZZ[p]$ given by
\begin{eqnarray*}
{\rm i)} \quad
[\Vc_{w_k}] &=& (p-1)(p^2-1)\cdots(p^{k-1}-1) \lambda_1^{k-1} \quad
\hbox{\rm if $k\leq m-1$, } \\
{\rm ii)} \quad [\Vc_{w_{m+1}}]
&=& (p-1)(p^2-1)\cdots(p^{m-1}-1)
\lambda_1^{m-1},\\
{\rm iii)} \quad
[\Vc_{w_{m+k}}] &=&
\frac{1}{2}
\frac{\prod_{i=1}^{m-1}(p^i-1)\prod_{i=m-k+2}^m(p^i+1)}
{\prod_{i=1}^{k-2}(p^i+1)\prod_{i=1}^{k-1}(p^i-1)}
\lambda_1^{m+k-2} \quad
\hbox{\rm if $2\leq k\leq m$. }
\end{eqnarray*}
Furthermore, we have that $\Vc_{w_m}=\emptyset$.
\end{theorem}

Finally in the twisted even case we have:

\begin{theorem}\label{Twistedclasses}
The cycle classes of the final strata $\Vc_w$ for twisted final elements
$w=w_j \in W_m^D s_m^{\prime}$ on the base $S$ are
powers in $\lambda_1$ with coefficients that are polynomials
in $\frac{1}{2}\ZZ[p]$ given by
\begin{eqnarray*}
{\rm i)} \quad
[\Vc_{w_k}] &=& (p-1)(p^2-1)\cdots(p^{k-1}-1) \lambda_1^{k-1} \quad
\hbox{\rm if $k\leq m-1$, } \\
{\rm ii)} \quad [\Vc_{w_m}]
&=& (p-1)(p^2-1)\cdots(p^{m}-1)
\lambda_1^{m-1},\\
{\rm iii)} \quad
[\Vc_{w_{m+k}}] &=&
\frac{1}{2}
\frac{\prod_{i=1}^{m}(p^i-1)\prod_{i=m-k+2}^{m-1}(p^i+1)}
{\prod_{i=1}^{k-1}(p^i+1)\prod_{i=1}^{k-2}(p^i-1)}
\lambda_1^{m+k-2} \quad
\hbox{\rm if $2\leq k\leq m$.}
\end{eqnarray*}
Furthermore, we have $\Vc_{w_{m+1}}=\emptyset$.
\end{theorem}
\section{Irreducibility}
In this section we shall show that about half of the $2m$ 
strata ${\mathcal V}_{w_i}$ on our moduli space ${\mathcal F}_N$ of 
$N$-polarized K3 surfaces are irreducible ($m$ strata in the B-case, $m-1$
in the D-case).

\begin{theorem} Let $p\geq 3$ and assume that ${\mathcal F}_N$
is the moduli space of primitively $N$-polarized K3 surfaces where 
$N^{\vee}/N$ has order prime to $p$.
If $w \in W_m^B$ (resp. $w \in W^D_{m}$ or $w \in W_m^Ds^{\prime}_m$) is a (twisted) final element with 
length $\ell(w)\geq m$, then the locus $\overline{\mathcal V}_{w}$
in ${\mathcal F}_N$ is irreducible.
\end{theorem}
\begin{proof}
(We do the $B$-case, leaving the other case to the reader.)
The idea behind the proof is to show that for $1 \leq i \leq m$ the stratum
$\overline{\mathcal U}_{w_i}$ is connected in the flag space ${\mathcal B}_N$.
Note that ${\mathcal F}_N$ is connected by our assumptions.
By Theorem \ref{height1structure} the stratum $\overline{\mathcal U}_{w_i}$
is normal, so if it is connected it must be irreducible. 
But then its image
$\overline{\mathcal V}_{w_i}$ in ${\mathcal F}_N$ is irreducible as well.
This shows the
advantage of working on the flag space. 

To show that $\overline{\mathcal U}_{w_i}$ is connected in ${\mathcal B}_N$
we use that its $1$-skeleton is connected, that is, that the union of the
$1$-dimensional strata that it contains, is connected and that every
irreducible component of $\overline{\mathcal U}_{w_i}$ intersects the
$1$-skeleton. To do that we 
prove the following facts:
\begin{enumerate}
\item{} The loci $\overline{\mathcal V}_{w_i}$ in ${\mathcal F}_N$ are
connected for $i<2m$ (that is, for $w_i\neq 1$).
\item{} Any irreducible component of any $\overline{\mathcal U}_{w}$ 
contains a point of ${\mathcal U}_1$.
\item{} The union $\cup_{i=2}^m \, \overline{\mathcal U}_{s_i}$ intersected
with a fibre of ${\mathcal B}_N \to {\mathcal F}_N$ over a point of the
superspecial locus ${\mathcal V}_{1}$ is connected.
\item{} For $1 \leq i \leq m$ the locus $\overline{\mathcal U}_{w_i}$ 
contains $\cup_{i=1}^m \overline{\mathcal U}_{s_i}$.
\end{enumerate}

In the proof we use the fact that the closure of strata on the flag
space is given by the Bruhat order in the Weyl group: ${\mathcal U}_v$
occurs in the closure of ${\mathcal U}_w$ if $v\geq w$ in the Bruhat order.
Furthermore, we observe that one knows by \cite{M}, \cite{P1}  
that $\lambda_1$ is an ample class. 

Together (1) and (3) will prove that the locus 
$\cup_{i=1}^m \overline{\mathcal U}_{s_i}$ (whose image in ${\mathcal F}_N$
is $\overline{\mathcal V}_{2m-1}$) is connected. We begin by proving (1).

Sections of a sufficiently high multiple of $\lambda_1$ embed ${\mathcal F}_N$
into projective space and we take its closure $\overline{\mathcal F}_N$.
By the result of Theorem \ref{Oddcycleclasses} 
(resp.\ \ref{Untwistedclasses} and \ref{Twistedclasses}) we know that the cycle
class $[\overline{\mathcal V}_{w_i}]$ is a multiple of $\lambda_1^{i-1}$,
so these loci are connected in $\overline{\mathcal F}_N$ for $i-1 <
\dim{\mathcal F}_N$. 
In particular, the $1$-dimensional 
locus $\overline{\mathcal V}_{w_{2m-1}}$ in ${\mathcal F}_N$ 
(which equals its closure in $\overline{\mathcal F}_N$) is connected.
On any irreducible component of $Y$ of $\overline{\mathcal V}_{w_i}$
in $\overline{\mathcal F}_N$ 
with $i <2m-1$ the intersection with $\overline{\mathcal V}_{w_{2m-1}}$
is cut out by a multiple of a positive power of
$\lambda_1$, hence it intersects this locus (in ${\mathcal F}_N$).
Since ${\mathcal V}_{w_i}$ contains $\overline{\mathcal V}_{w_{2m-1}}$
for $i<2m-1$ the connectedness follows. 

To prove (4) consider the reduced expression for $w_i$ for $i\leq m$: 
it is $s_is_{i+1}\cdots s_m s_{m-1}\cdots s_1$, see \cite[Lemma 11.1]{E-vdG2}. 
This shows that all
the $s_i$ occur in it and we see that the $\overline{\mathcal U}_{s_i}$ 
for $i=1,\ldots,m$
occur in the closure of $\overline{\mathcal U}_{w_i}$.

The proof of (2) is similar to the 
proof of \cite[Prop.\ 6.1]{E-vdG1} and uses induction on the Bruhat order. 
If $\ell(w) \leq 2m-2$ then $\overline{\mathcal U}_w$ is proper in
${\mathcal B}_N$. If an irreducible component has 
a non-empty intersection with  a $\overline{\mathcal U}_{w'}$ 
with $w'>w$, then induction provides a point of ${\mathcal U}_1$;
otherwise we can apply a version of the Raynaud trick as in
\cite[Lemma 6.2]{E-vdG1} and conclude that $w=1$. 
If $\ell(w)=2m-1$ then the image of 
any irreducible component $Y$ of $\overline{\mathcal U}_w$
in ${\mathcal F}_N$ is either contained in $\overline{\mathcal V}_{w_3}$ 
and then $Y$ is proper in ${\mathcal B}_N$ 
or the image coincides with $\overline{\mathcal V}_{w_2}$.
In the latter case it maps in a generically finite way to it and therefore
any irreducible component $Y$ of $\overline{\mathcal U}_w$ 
intersects the fibres over the superspecial points, hence by induction
contains a point of~${\mathcal U}_1$. 

For (3) we now look in the fibre $Z$ of the flag space over the 
image of a point of ${\mathcal U}_{1}$. 
This corresponds to a K3 surfaces for which the Hodge
filtration $U_{-1}\subset U_0 \subset U_1 \subset U_2=H$
coincides with the conjugate filtration
$U_{-1}^c\subset U_0^c \subset U_1^c \subset U_2^c=H$.
Moreover, we have the identifications
$$
F^*(U_0)\cong (U_2^c/U^c_1)=(U_0^c)^{\vee}= U_0^{\vee},
$$
given by Cartier and the intersection pairing and similarly
$$
F^*(U_1/U_0)\cong (U_1/U_0)^{\vee},
$$
giving $U_0$ and $U_1/U_0$ (and also $U_2/U_1)$) the structure of a $p$-unitary space. Indeed,
if $S$ is an ${\FF}_p$-scheme then a $p$-unitary vector bundle ${\mathcal E}$ over $S$ is a
vector bundle together with an isomorphism $F^*({\mathcal E})\cong {\mathcal E}^*$ with $F$ the absolute Frobenius. 
This gives rise to a bi-additive map 
$\langle \, , \, \rangle: {\mathcal E}\times {\mathcal E} \to {\mathcal O}_S$
satisfying 
$\langle f x, y\rangle = f^p\langle x, y\rangle$ 
and 
$\langle x,fy\rangle = f \langle x, y \rangle$ for $f$ a section 
of ${\mathcal O}_S$.
In the \'etale topology this notion is equivalent 
to a local system of ${\FF}_{p^2}$-vector spaces,
cf.\ \cite[Prop.\ 7.2]{E-vdG1}.

In case that $S={\rm Spec}({\FF}_{p^2})$ we can then consider the
flag space $Z$ of complete $p$-unitary flags on $U_1/U_0$.
The smallest $p$-unitary stratum there is the stratum of flags that
coincide with their $p$-unitary dual.
Such flags are defined over ${\FF}_{p^2}$
as one sees by taking the dual once more.

So we look now at self-dual flags $A_{\bullet}= \{ A_1 \subset A_2 \subset 
\cdots \subset A_{2m-1}\}$ 
on $A=U_1/U_0$ and we compare
the flag $A_{\bullet}^{(p)}$ with the flag $A_{\bullet}$. 
(Note that the indices $i$ now run from $1$ to $m-1$ instead of from $1$ to $m$
since we leave $U_0$ fixed.) 
For an element $s=(i, i+1)$ with $1 \leq i \leq m-2$ we look at 
the intersection of the stratum $\overline{\mathcal U}_{s}$ with the fibre $Z$.
It consists of those flags $A_{\bullet}$ with the property that the 
steps $A_j$ for $j\neq i$ and $j\neq 2m-1-i$ are ${\FF}_{p^2}$-rational
and that for all $j$ we have $A_{j}= A_{2m-1-j}^{\bot}$.
We see that we can choose $A_i$ freely by prescribing its image in
$A_{i+1}/A_{i-1}$, hence this locus is a ${\PP}^1$. 
In case $s=(m-1, m+1)$ we have to choose a space $A_{m-1}$ and its orthogonal 
$A_m=A_{m-1}^{\bot}$
in $A_{m-1} \subset A_{m-1}\subset A_m \subset A_{m+1}$. Now all
non-degenerate $p$-unitary forms are equivalent, so we may choose the
form $x^{p+1}+y^{p+1}+z^{p+1}$ on the $3$-dimensional space $A_{m+1}/A_{m-2}$.
So in this case $\overline{\mathcal U}_s$ is isomorphic to the Fermat curve.
The points of ${\mathcal U}_1$ are the ${\FF}_{p^2}$-rational points on it.

In case the space $A=U_1/U_0$ is even-dimensional, say $\dim A=2m-2$
the same argument works for $s_i$ with $1\leq i \leq m-2$. 
For $s=(m-2, m)(m-1, m+1)$ (resp.\ for $s=(m-1,m)$) 
we remark that it equals $s' s_{m-2}s'$ with $s'=(m-1, m)$,
hence we find a ${\PP}^1$ by picking a flag $A_{m-2} \subset
A_{m-1} \subset A_{m}$. 

This shows that we remain in the same connected component of 
$\overline{\mathcal U}^{(1)} \cap Z$, 
with ${\mathcal U}^{(1)}$ the union of the $1$-dimensional strata ${\mathcal U}_v$, if we change the flag $A_{\bullet}$ at
place $i$ and $2m-1-i$ compatibly. By Lemma 7.6 of \cite{E-vdG1}
this implies that $\overline{\mathcal U}^{(1)} \cap Z$ is connected.

\end{proof}

\end{document}